\documentclass[leqno]{amsart}
\usepackage{amsfonts}
\usepackage{amsmath}
\usepackage{amscd}
\usepackage{amssymb}
\usepackage{amstext}

\newtheorem{theorem}{Theorem}[section]
\theoremstyle{plain}
\newtheorem{corollary}[theorem]{Corollary}
\newtheorem{definition}[theorem]{Definition}

\numberwithin{equation}{section}

\begin{document}
\title[Alzer Inequality for Hilbert Spaces Operators]{Alzer Inequality for Hilbert Spaces Operators}
\author[A. Morassaei]{Ali Morassaei}
\address[A. Morassaei and F. Mirzapour]{Department of Mathematics, Faculty of Sciences, University of Zanjan, P. O. Box 45195-313, Zanjan, Iran}
\email{morassaei@znu.ac.ir, f.mirza@znu.ac.ir}
\author[F. Mirzapour]{Farzollah Mirzapour}
\thanks{2010\textit{\ AMS Math. Subject Classification. }Primary 47A63; Secondary 15A42, 46L05, 47A30.}

\begin{abstract}
In this paper, we give the Alzer inequality for Hilbert space operators as follows:

Let $A, B$ be two selfadjoint operators on a Hilbert space $\mathcal H$ such that $0 < A, B \le \frac{1}{2}I$, where $I$ is identity operator on $\mathcal H$. Also, assume that $A \nabla_\lambda B:=(1-\lambda)A+\lambda B$ and $A \sharp_\lambda B:=A^{\frac{1}{2}}\left(A^{-\frac{1}{2}}BA^{-\frac{1}{2}}\right)^\lambda A^{\frac{1}{2}}$ are arithmetic and geometric means of $A, B$, respectively, where $0 < \lambda < 1$. We show that if $A$ and $B$ are commuting, then
$$
B'~\nabla_\lambda~A' - B'~\sharp_\lambda~A' \le A~\nabla_\lambda~B - A~\sharp_\lambda~B\,,
$$
where $A':=I-A$, $B':=I-B$ and $0 < \lambda \le \frac{1}{2}$. Also, we state an open problem for an extension of Alzer inequality.
\end{abstract}

\maketitle

\section{Introduction and preliminaries}
Let $x_1,\cdots,x_n\in (0, \frac{1}{2}]$ and $\lambda_1,\cdots,\lambda_n>0$ with $\sum_{j=1}^n\lambda_j=1$. We denote by $A_n$ and $G_n$,
the arithmetic and geometric means of $x_1,\cdots,x_n$ respectively, i.e
$$
A_n=\sum_{j=1}^n\lambda_jx_j,\quad\quad G_n=\prod_{j=1}^nx_j^{\lambda_j}\,,
$$
and also by $A'_n$ and $G'_n$, the arithmetic and geometric means of $1-x_1,\cdots,1-x_n$ respectively, i.e.
$$
A'_n=\sum_{j=1}^n\lambda_j(1-x_j),\quad\quad G'_n=\prod_{j=1}^n(1-x_j)^{\lambda_j}\,.
$$
Alzer proved the following inequality and its refinement \cite{A1, A2}
\begin{equation}\label{e1.1}
A'_n-G'_n \le A_n - G_n.
\end{equation}

Throughout the paper, let $\mathbb{B}(\mathcal H)$ denote the
algebra of all bounded linear operators acting on a complex Hilbert
space $(\mathcal H,\langle \cdot,\cdot\rangle)$ and $I$ is the identity
operator. In the case when $\dim \mathcal{H} = n$, we identify
$\mathbb{B}(\mathcal H)$ with the full matrix algebra $\mathcal{M}_n(\mathbb{C})$
 of all $n \times n$ matrices with entries in the complex field and denote its identity
 by $I_n$. A selfadjoint operator $A\in\mathbb{B}(\mathcal{H})$ is called
positive (strictly positive) if $\langle Ax,x\rangle\ge0~~(\langle Ax,x\rangle > 0)$ holds for every $x\in \mathcal{H}$
and then we write $A\ge0~(A>0)$ \cite{B, FMPS}. For every selfadjoint operators $A,B\in\mathbb{B}(\mathcal{H})$, we say $A\le B$ if $B-A \ge 0$.
Let $f$ be a continuous real valued function defined on an interval
$[\alpha, \beta]$. The function $f$ is called operator decreasing if
$B\le A$ implies $f(A)\le f(B)$ for all $A,B$ with spectra in
$[\alpha, \beta]$. A function $f$ is said to be operator concave on
$[\alpha, \beta]$ if
$$
\lambda f(A) + (1 - \lambda)f(B)\le f(\lambda A + (1 - \lambda)B)
$$
for any selfadjoint operators $A,B\in\mathbb{B}(\mathcal{H})$ with spectra in $[\alpha,
\beta]$ and all $\lambda\in[0,1]$.

\medskip

The main result of this paper is the following theorem:

\medskip

\textbf{Theorem}~(Alzer Inequality).
Suppose that $A, B\in \mathbb{B}(\mathcal{H})$ are commuting operators such that $0 < A \le B \le \frac{1}{2}I$, and let $A':=I-A$ and $B'=I-B$. If $0 < \lambda \le \frac{1}{2}$, then
$$
B'~\nabla_\lambda~A' - B'~\sharp_\lambda~A' \le A~\nabla_\lambda~B - A~\sharp_\lambda~B\,.
$$

\section{Main results}
In this section, we state an identity between arithmetic and geometric mean for positive operators and then we consequent the Alzer inequality.

We recall that, the \textit{weighted arithmetic mean} $\nabla_\lambda$ and the \textit{weighted geometric mean} (the \textit{$\lambda$-power mean}) $\sharp_\lambda$ defined for $0<\lambda<1$:
\begin{align*}
A~\nabla_\lambda~B&:=(1-\lambda)A+\lambda B\,,\\
A~\sharp_\lambda~B&:=A^{\frac{1}{2}}\left(A^{-\frac{1}{2}}BA^{-\frac{1}{2}}\right)^\lambda A^{\frac{1}{2}}\,.
\end{align*}
Also, we know that $A ~\sharp_\lambda~ B=B ~\sharp_{1-\lambda}~ A$ and if $AB=BA$ then $A \sharp_\lambda B=A^{1-\lambda}B^\lambda$.

Notice that if $\lambda=\frac{1}{2}$ in above definitions, we have the classic arithmetic and geometric means and denote its as follows:
\begin{align*}
{\mathbf A}&:=A~\nabla~B=A~\nabla_{\frac{1}{2}}~B=\frac{1}{2}A+\frac{1}{2}B\,,\\
{\mathbf G}&:=A~\sharp~B=A~\sharp_{\frac{1}{2}}~B=A^{\frac{1}{2}}\left(A^{-\frac{1}{2}}BA^{-\frac{1}{2}}\right)^{\frac{1}{2}}A^{\frac{1}{2}}\,.
\end{align*}
Also, we know that ${\mathbf A'}=A'~\nabla~B'$ and ${\mathbf G'}=A'~\sharp~B'$.

In the following theorem, we state distance between the arithmetic mean and the geometric mean as an infinite series.
%----------------------------------------------------------------------------------
\begin{theorem}\label{t3.1}
Assume that $A$ and $B$ are two positive operators in $\mathbb{B}(\mathcal{H})$ such that $\|B^{-\frac{1}{2}}AB^{-\frac{1}{2}}\|<1$ and $\lambda \in (0, 1)$. Then we have
\begin{equation}\label{e3.1}
A ~\nabla_{\lambda} ~B - A~ \sharp_{\lambda} ~B=\sum_{k=2}^\infty (-1)^{k-1}\binom{1-\lambda}{k}\left(AB^{-1}-I\right)^kB\,.
\end{equation}
\end{theorem}
\begin{proof}
By using the binomial series, we have
\begin{align}\label{e3.2}
\left(B^{-\frac{1}{2}}AB^{-\frac{1}{2}}\right)^{1-\lambda}&=\left(I+\left(B^{-\frac{1}{2}}AB^{-\frac{1}{2}}-I\right)\right)^{1-\lambda}\nonumber\\
&=I+\sum_{k=1}^\infty\binom{1-\lambda}{k}\left(B^{-\frac{1}{2}}AB^{-\frac{1}{2}}-I\right)^k\,.
\end{align}
Now, by multiplying each side \eqref{e3.2} by $B^{\frac{1}{2}}$, we get
\begin{align*}
B^{\frac{1}{2}}\Big(B^{-\frac{1}{2}}&AB^{-\frac{1}{2}}\Big)^{1-\lambda}B^{\frac{1}{2}}\\
&=B+\sum_{k=1}^\infty \binom{1-\lambda}{k}B^{\frac{1}{2}}\left(B^{-\frac{1}{2}}AB^{-\frac{1}{2}}-I\right)^kB^{\frac{1}{2}}\\
&=B+\binom{1-\lambda}{1}(A-B)+\sum_{k=2}^\infty\binom{1-\lambda}{k} B^{\frac{1}{2}}\left(B^{-\frac{1}{2}}AB^{-\frac{1}{2}}-I\right)^kB^{\frac{1}{2}}\\
&=B+(1-\lambda)(-1)(B-A)+\sum_{k=2}^\infty\binom{1-\lambda}{k} B^{\frac{1}{2}}\left[B^{-\frac{1}{2}}(A-B)B^{-\frac{1}{2}}\right]^kB^{\frac{1}{2}}\\
&=(1-\lambda)A+\lambda B-\sum_{k=2}^\infty (-1)^{k-1} \binom{1-\lambda}{k} \left[(A-B)B^{-1}\right]^kB\,,
\end{align*}
so, $B~ \sharp_{1-\lambda} ~A=A ~\nabla_{\lambda}~ B-\sum_{k=2}^\infty (-1)^{k-1} \binom{1-\lambda}{k} \left(AB^{-1}-I\right)^kB$, which completes the proof.
\end{proof}
%-------------------------------------------------------------------------------------------------
We know that, if $A$ and $B$ are two commuting positive operators in $\mathbb{B}(\mathcal{H})$, then $AB$ is positive operator and $(AB)^{\frac{1}{2}}=A^{\frac{1}{2}}B^{\frac{1}{2}}$. Furthermore, if $B$ is invertible, then $AB^{-1}=B^{-1}A$. Also, we recall that if $A$ and $B$ are not commuting, then $AB$ is not necessarily positive. For example, $A=\left(\begin{array}{cc} 1 & 0 \\ 0 & 0 \end{array}\right)$ and $B=\left(\begin{array}{cc} 1 & 1 \\ 1 & 1 \end{array}\right)$ are positive but their product is not \cite[p. 309]{M}.

Now, by using the above statements and Theorem \ref{t3.1}, the following corollary is obvious.
%---------------------------------------------------------------------------------------------------
\begin{corollary}\label{c2.3}
With the assumptions in Theorem \ref{t3.1}, if $A$ and $B$ are commuting, then
$$
A ~\nabla_{\lambda} ~B - A~ \sharp_{\lambda} ~B=\sum_{k=2}^\infty (-1)^{k-1}\binom{1-\lambda}{k}B^{\frac{1-k}{2}}(B-A)^kB^{\frac{1-k}{2}}\,.
$$
\end{corollary}
%---------------------------------------------------------------------------------------------------
In the following theorem we state the Alzer inequality for two commuting positive operator in $\mathbb{B}(\mathcal{H})$.
\begin{theorem}[Alzer Inequality]\label{t3.2}
Suppose that $A, B\in \mathbb{B}(\mathcal{H})$ are commuting operators such that $0 < A \le B \le \frac{1}{2}I$, and let $A':=I-A$ and $B'=I-B$. If $0 < \lambda \le \frac{1}{2}$, then
\begin{equation}\label{e3.3}
B'~\nabla_\lambda~A' - B'~\sharp_\lambda~A' \le A~\nabla_\lambda~B - A~\sharp_\lambda~B\,.
\end{equation}
\end{theorem}
\begin{proof}
It is clear that $0 < A \le B \le \frac{1}{2}I \le B' \le A' < I$ and also $A'B'=B'A'$. By using Corollary \ref{c2.3}, we obtain
\begin{equation}\label{e3.4}
A ~\nabla_{\lambda} ~B - A~ \sharp_{\lambda} ~B=\sum_{k=2}^\infty (-1)^{k-1}\binom{1-\lambda}{k}B^{\frac{1-k}{2}}(B-A)^kB^{\frac{1-k}{2}}\,,
\end{equation}
and
\begin{equation}\label{e3.5}
B' ~\nabla_{\lambda} ~A' - B'~ \sharp_{\lambda} ~A'=\sum_{k=2}^\infty (-1)^{k-1}\binom{1-\lambda}{k}A'^{\frac{1-k}{2}}(A'-B')^kA'^{\frac{1-k}{2}}\,.
\end{equation}
Since $A'-B'=B-A$, $B \le A'$ and $k \ge 2$ we have $A'^{\frac{1-k}{2}}(A'-B')^kA'^{\frac{1-k}{2}} \le B^{\frac{1-k}{2}}(B-A)^kB^{\frac{1-k}{2}}$. On the other hand, since $0 < \lambda \le \frac{1}{2}$ and $(-1)^{k-1}\binom{\alpha}{k} > 0$ for all $0 < \alpha < 1$ and $k \ge 2$, we get $(-1)^{k-1}\binom{1-\lambda}{k}A'^{\frac{1-k}{2}}(A'-B')^kA'^{\frac{1-k}{2}} \le (-1)^{k-1}\binom{1-\lambda}{k}B^{\frac{1-k}{2}}(B-A)^kB^{\frac{1-k}{2}}$, which completes the proof.
\end{proof}
%-------------------------------------------------------------------------------------------------
\begin{corollary}
With the above notations, we have
$$
{\mathbf A}'-{\mathbf G}'\le {\mathbf A}-{\mathbf G}.
$$
\end{corollary}
\begin{proof}
Sufficient in the Theorem \ref{t3.2} we set $\lambda=\frac{1}{2}$ and use of this fact that $A \nabla  B=B \nabla A$ and $A \sharp B=B \sharp A$.
\end{proof}
%-------------------------------------------------------------------------------------------------
\section{Open problem}
In this section, we present an extension of Alzer inequality for Hilbert space
operators as an open problem. For this purpose, first, we express some fundamental properties of the geometric mean. For to see many details c.f. \cite{A, ALM, JLY, RLC}.

The geometric mean $\mathbf G_2:=\mathbf G_2(A, B)$ of two positive operators $A$ and $B$ was introduced as the solution of the matrix optimization problem, \cite{A}
\begin{equation}\label{e1.1}
\mathbf G_2(A, B):=\max\left\{X~:~X^*=X,\quad\left(\begin{array}{cc}A & X\\ X & B\end{array}\right)\ge 0\right\}\,.
\end{equation}
This operator mean can be also characterized as the strong limit of the arithmetic-harmonic sequence $\{\Phi_n(A, B)\}$ defined by \cite{AR, FF}
\begin{equation}\label{e1.2}
\begin{cases}
\Phi_0(A, B)=\frac{1}{2}A+\frac{1}{2}B\,,\\
\Phi_{n+1}(A, B)=\frac{1}{2}\Phi_n(A, B)+\frac{1}{2}A(\Phi_n(A, B))^{-1}B\quad(n\ge 0)\,.
\end{cases}
\end{equation}
We know that, the explicit form of $\mathbf G_2(A, B)$ is given by
\begin{equation}\label{e1.3}
\mathbf G_2(A, B)=A^{\frac{1}{2}}\left(A^{-\frac{1}{2}}BA^{-\frac{1}{2}}\right)^{\frac{1}{2}}A^{\frac{1}{2}}\,.
\end{equation}
M. Ra\"{i}ssouli, F. Leazizi and M. Chergui in \cite{RLC} described an extended algorithm of \eqref{e1.2} involving several positive operators. The main idea of such an extension comes from the fact that the arithmetic, harmonic and geometric means of $m$ positive real numbers $a_1, a_2, \cdots, a_m$ can be written recursively as follows
\begin{equation}\label{e1.4}
\mathbf A_m(a_1, a_2, \cdots, a_m):=\frac{1}{m}\sum_{j=1}^ma_j=\frac{1}{m}a_1+\frac{m-1}{m}\mathbf A_{m-1}(a_2, \cdots, a_m)\,,
\end{equation}
\begin{equation}\label{e1.5}
\mathbf H_m(a_1, a_2, \cdots, a_m):=\left(\frac{1}{m}\sum_{j=1}^ma_j^{-1}\right)^{-1}=\left(\frac{1}{m}a_1^{-1}+\frac{m-1}{m}\mathbf H_{m-1}(a_2, \cdots, a_m)\right)^{-1}\,,
\end{equation}
\begin{equation}\label{e1.6}
\mathbf G_m(a_1, a_2, \cdots, a_m):=\sqrt[m]{a_1a_2 \cdots a_m}=a_1^{\frac{1}{m}}\left(\mathbf G_{m-1}(a_2, \cdots, a_m)\right)^{\frac{m-1}{m}}\,.
\end{equation}
The extensions of \eqref{e1.4} and \eqref{e1.5} when the scalers variable $a_1, a_2, \cdots, a_m$ are positive operators can be immediately given, by setting $A^{-1}=\lim_{\epsilon\downarrow0}(A+\epsilon I)^{-1}$. We know that the power geometric mean of two positive operators $A$ and $B$ defined by
\begin{equation}\label{e1.7}
\Phi_{\frac{1}{m}}(A, B):=B^{\frac{1}{2}}\left(B^{-\frac{1}{2}}AB^{-\frac{1}{2}}\right)^{\frac{1}{m}}B^{\frac{1}{2}}\,.
\end{equation}

Assume that $A_1,\cdots,A_m\in\mathbb{B}(\mathcal{H})\quad (m \ge 2)$ are $m$ positive operators. In this section we introduce the geometric mean of $A_1,\cdots,A_m$. By using the algorithm \eqref{e1.2}, we define the recursive sequence $\{T_n\}:=\{T_n(A, B)\}$, where $A, B\in\mathbb{B}(\mathcal{H})$ are two positive operators, as follows
\begin{equation}\label{e2.1}
\begin{cases}
T_0=\frac{1}{m}A+\frac{m-1}{m}B\,;\\
T_{n+1}=\frac{m-1}{m}T_n+\frac{1}{m}A(T_n^{-1}B)^{m-1}\quad(n \ge 0)\,.
\end{cases}
\end{equation}
In what follows, for simplicity we write $\{T_n\}$ instead of $\{T_n(A, B)\}$ and we set
$$
T_n^{(-1)}=\left(T_n(A^{-1}, B^{-1})\right)^{-1}\,.
$$
In the following theorem Ra\"{i}ssouli, Leazizi and Chergui \cite{RLC} proved the convergence of the operator sequence $\{T_n\}$.
\begin{theorem}\label{t2.1}
With the above assumptions, the sequence $\{T_n\}:=\{T_n(A, B)\}$ converges decreasingly in $\mathbb{B}(\mathcal{H})$, with the limit
\begin{equation}\label{e2.2}
\lim_{n\uparrow+\infty}T_n:=\Phi_{\frac{1}{m}}(A, B)=B^{\frac{1}{2}}\left(B^{-\frac{1}{2}}AB^{-\frac{1}{2}}\right)^{\frac{1}{m}}B^{\frac{1}{2}}\,.
\end{equation}
Further, the next estimation holds
\begin{equation}\label{e2.3}
0 \le T_n-\Phi_{\frac{1}{m}}(A, B) \le \left(1-\frac{1}{m}\right)^n\left(T_0-T_0^{(-1)}\right)\quad\forall n \ge 0\,.
\end{equation}
\end{theorem}
Notice that $\Phi_{\frac{1}{m}}(A, B)=A^{\frac{1}{m}}B^{1-\frac{1}{m}}$ when $A$ and $B$ are two commuting positive operators and so, $\Phi_{\frac{1}{m}}(A, I)=A^{\frac{1}{m}}$, $\Phi_{\frac{1}{m}}(I, B)=B^{1-\frac{1}{m}}$ for all positive operators $A$ and $B$. Also, the map $(A, B)\mapsto \Phi_{\frac{1}{m}}(A, B)$ satisfies the conjugate symmetry relation, i.e.
\begin{equation}\label{e2.4}
\Phi_{\frac{1}{m}}(A, B)=A^{\frac{1}{2}}\left(A^{-\frac{1}{2}}BA^{-\frac{1}{2}}\right)^{\frac{m-1}{m}}A^{\frac{1}{2}}=\Phi_{\frac{m-1}{m}}(B, A)\,.
\end{equation}
In the same paper, we see the definition of geometric operator mean of $A_1,\cdots,A_m$ as follows.
\begin{definition}\label{d2.1}
Assume that $A_1,\cdots,A_m\in\mathbb{B}(\mathcal{H})$ are the positive operators. The geometric operator mean of $A_1,\cdots,A_m$ is defined by the relationship
\begin{equation}\label{e2.5}
\mathbf G_m(A_1, A_2, \cdots, A_m)=\Phi_{\frac{1}{m}}\left(A_1,\mathbf G_{m-1}(A_2, \cdots, A_m)\right)\,.
\end{equation}
\end{definition}
It is easy to verify that, if $A_1,\cdots,A_m$ are commuting, then
$$
\mathbf G_m(A_1, A_2, \cdots, A_m)=\left(A_1, A_2 \cdots A_m\right)^{\frac{1}{m}}\,.
$$
In particular, for all positive operators $A\in\mathbb{B}(\mathcal{H})$ we have $\mathbf G_m(A, A, \cdots, A)=A$ and $\mathbf G_m(I, I, \cdots, A, I, \cdots, I)=A^{\frac{1}{m}}$. Also, we know that $(A, B)\mapsto \mathbf G_2(A, B)$ is symmetric, but $\mathbf G_m$ is not symmetric for $m \ge 3$, for more details see \cite[Example 2.3]{RLC}.

The geometric operator mean $\mathbf G_m(A_1, A_2, \cdots, A_m)$ has nice properties that for seeing more details c.f. \cite{RLC}.

\textbf{Open Problem}. Let $A_1,\cdots,A_n$ be $n$ selfadjoint operators on an Hilbert space $\mathcal H$ such that $0<A_j\le\frac{1}{2}I$, where $I$ is identity operator on $\mathcal H$ \cite{B, FMPS}. Also, let ${\mathbf A}_n:={\mathbf A}_n(A_1,\cdots,A_n)$ and ${\mathbf G}_n:={\mathbf G}_n(A_1,\cdots,A_n)$ be arithmetic and geometric means of $A_1,\cdots,A_n$ \cite{RLC}, and ${\mathbf A}'_n:={\mathbf A}_n(A'_1,\cdots,A'_n)$ and ${\mathbf G}'_n:={\mathbf G}_n(A'_1,\cdots,A'_n)$ be arithmetic and geometric means of $A'_1,\cdots,A'_n$ where $A'_j:=I-A_j\quad(j=1,\cdots,n)$, respectively. Then it seems that
$$
{\mathbf A}'_n-{\mathbf G}'_n\le {\mathbf A}_n-{\mathbf G}_n.
$$

%-------------------------------------------------------------------------------------------------

\end{document}